\documentclass[12pt]{article}
\usepackage{amssymb,latexsym,amsmath}
\usepackage{amsthm}
\usepackage{epsfig}
\date{}

\newcommand{\C}{\mathbb{C}}
\newcommand{\Z}{\mathbb{Z}}
\newcommand{\N}{\mathbb{N}}
\newcommand{\R}{\mathbb{R}}

\renewcommand{\deg}{{\rm deg} \,}

\newcommand{\nor}[1]{\parallel {#1} \parallel}
\newcommand{\mo}[1]{\left|{#1}\right|}
\newcommand{\vs}{\vspace{.15in}}

\newcommand{\conj}[1]{\left\{ {#1}\right\}}
\newcommand{\dps}[1]{\displaystyle{#1}}

\newcommand{\comp}[1]{\overline{#1}}

\newcommand{\cc}{\mathbb{C}}
\newcommand{\dd}{\mathbb{D}}
\newcommand{\D}{\mathbb{D}}

\def \R{I\kern-2.5ptR}

\def \Z{{Z\kern-5.0ptZ\,}}

\numberwithin{equation}{section}

\newtheorem{theorem}{Theorem}[section]
\newtheorem{lemma}[theorem]{Lemma}
\newtheorem{corollary}[theorem]{Corollary}
\newtheorem{proposition}[theorem]{Proposition}

\newtheorem{Definition}[theorem]{Definition}
\newenvironment{definition}{\begin{Definition}\rm}{\end{Definition}}
\newtheorem{Remark}[theorem]{Remark}
\newenvironment{remark}{\begin{Remark}\rm}{\end{Remark}}
\newtheorem{Example}[theorem]{Example}
\newenvironment{example}{\begin{Example}\rm}{\end{Example}}
\newtheorem{Problem}[theorem]{Problem}

\def\deg{\mbox{{\em deg}}\,}
\def\vect#1{\mbox{\em vec}(#1)}
\def\calR#1{\mathbb C_{#1}(z)}
\newcommand{\epsi}[2]{ \def\tst{#2} \ifx\tst\empty\epsilon_{#1}\else\epsilon_{#2,#1}\fi }

\makeatletter
\renewenvironment{proof}[1][\proofname]{\par
\pushQED{\qed}%
\normalfont \topsep6\p@\@plus6\p@\relax
\trivlist
\item\relax
{\itshape
#1\@addpunct{.}}\hspace\labelsep\ignorespaces
}{%
\popQED\endtrivlist\@endpefalse
}
\makeatother

\begin{document}
\title{On rational functions without Froissart doublets}
\author{Bernhard Beckermann\thanks{
Laboratoire Painlev\'e UMR 8524, UFR Math\'ematiques --
M3, Universit\'e de Lille, F-59655 Villeneuve d'Ascq CEDEX, France. E-mail:
{\tt $\{$bbecker,matos$\}$@math.univ-lille1.fr}. Supported in part by the Labex CEMPI  (ANR-11-LABX-0007-01).}, George Labahn\thanks{Cheriton School of Computer Science, University of Waterloo, Waterloo, Ontario, Canada. E-mail:
{\tt glabahn@uwaterloo.ca}},
and Ana C. Matos$^*$}
\date{}
\maketitle

\begin{abstract}
In this paper we consider the problem of working with rational functions in a numeric environment.  A particular problem when modeling with such functions is the existence of Froissart doublets, where a zero is close to a pole.  We discuss three different parameters which allow one to monitor the absence of Froissart doublets for a given general rational function. These include the euclidean condition number of an underlying Sylvester-type matrix, a parameter for determing coprimeness of two numerical polynomials and bounds on the spherical derivative.  We show that our parameters sharpen those found in a previous paper by two of the authors.
\end{abstract}



\section{Introduction}
Let $\cc [z]$ be the space of polynomials with complex coefficients, $\cc_n[z]$ the subset of polynomials of degree at most $n$, and
$$
\calR{m,n}=\conj{\frac{p}{q} ,\quad  p\in \cc_m [z], q\in \cc_n [z] , ~q \ne 0}
$$
the set of rational functions.
Rational functions have long played an important role in applied mathematics. 
As an example, Pad\'e approximants and rational interpolants are used for approximation, analytic continuation and for determining singularities of a function  \cite{G-M,T13}. Also, Pad\'e approximants of the $z$ transform of noisy signals are employed \cite{Bessis}
for detecting the number of significant signals, their frequencies, damping, phase and amplitude.
Other applications include sparse interpolation \cite{GLL,KY07}, computer algebra \cite{gerhard,driver} and exponential analysis \cite{cuyt,peelman,potts}.

In order to successfully use rational functions for modelling,  one first has to address the subtle question of choosing a priori the degrees $m,~n$. We want to make sure that the rational function $r\in \calR{m,n}$ is nondegenerate, that is, at least one of the numerator or denominator degrees is maximal after removing common factors. In addition, such a rational function should also be sufficiently ``far'' from $\calR{m-1,n-1}$. Indeed, overshooting the degree may produce strange artifacts commonly referred to as spurious poles. By this we mean we might have so-called Froissart doublets \cite{Froissart}, that is, a pair of points, one a pole and the other a zero of $r$, which are close to each other. This of course makes it impossible to approach smooth functions with such rational functions. Another second well-known artifact is a simple pole with small residual, which does not seem to be significant if one wants to evaluate $r$ at points not too close to such a pole.  These issues become particularly significant when computation is done in  a numeric environment where one can only obtain close rather than exact answers.

In a recent paper \cite{GGT12}, the authors introduced the notion of robust Pad\'e approximants, a lower order Pad\'e approximant based on the SVD of the underlying Toeplitz matrix. There the authors showed in many illustrating examples that their robust Pad\'e approximants  no longer have spurious poles. Only later was it shown that the underlying nonlinear Pad\'e map taking the coefficients of the initial Taylor series and mapping them to the coefficients in the basis of monomials of the numerator and denominator of a Pad\'e approximant is forward well-conditioned (but not necessarily backward) at such robust Pad\'e approximants \cite[Theorem~1.2]{BM}, and that robust Pad\'e approximants may have spurious poles \cite{masc,BM}. A first important contribution to this set of questions is \cite{WW} for the continuity of the Pad\'e map.

In computer algebra the area of symbolic/numeric computation often considers correctness and stability issues when working with polynomial arithmetic. For example there has been considerable work on problems such as the numerical gcd of two polynomials having floating point coefficients. However there seems to be very little work which deals with
numerical analysis around rational functions. Indeed even the first issue of clarifying how to measure distances in $\calR{m,n}$ has not really been considered.

It seems natural that one should expect a connection between ``nearly'' degenerate rational functions and numerators and denominators having a non-trivial numerical gcd. This includes the two papers \cite{BL,Cor} where coprimeness parameters are considered and which both make the link with the underlying Sylvester matrix formed by the coefficients of the numerator and denominator (see Definition~\ref{def_Sylvester} in the third section).

The aim of this paper is to discuss three different parameters which allow one to monitor the absence of Froissart doublets for a given general rational function $r\in \calR{m,n}$. 
 \begin{itemize}
    \item The euclidean condition number of underlying Sylvester-type matrices depending on some integer $\ell$;
    \item the coprimeness parameter of \cite{BL,Cor};
    \item bounds on the spherical derivative.
 \end{itemize}
In each case we will show how our first two parameters generalizes and sharpens  the parameters presented in \cite{BM}. In the case of the  spherical derivative, the two new parameters introduced here are essentially best Lipschitz constants of a rational function and we show how measuring distances in $\calR{m,n}$ also partially sharpens some distance measures found in \cite{BM}.

This paper is a follow up to the paper \cite{BM} where some of the same problems were considered. In order to explain our contributions in more detail, we describe results from the previous paper along with our new findings. For this  it is helpful to refer to different properties given in Figure~\ref{figure}.
The authors in \cite{BM} refer to $r=p/q$ with modest euclidean condition number of the underlying Sylvester type matrix for $\ell=1$ as well-conditioned rational functions, and deduced several properties of such functions. This includes, for example, the absence of Froissart doublets \cite[Theorem~1.3(a)]{BM} and of small residuals for simple poles \cite[Theorem~1.3(b)]{BM}, a large distance to the set $\calR{m-1,n-1}$ of degenerate rational functions \cite[Theorem~1.4]{BM}, but also a modest forward and backward condition number for the non-linear Pad\'e map for well-conditioned Pad\'e approximants \cite[Theorem~1.2]{BM}. They also establish the equivalence of two different distances in $\calR{m,n}$, one based on values and the other on coefficients of rational functions \cite[Theorem~4.1]{BM}. This corresponds to the implications on the right of Figure~\ref{figure}. We will establish later in Theorem~\ref{thm_norms}  that the choice of our parameter $\ell$ in Definition~\ref{def_Sylvester} is not essential.

The results in this paper correspond to the left side of Figure~\ref{figure}.
In  \S 2.1 we recall some of the findings from \cite{BM} and also the coprimeness parameter of \cite{BL,Cor}. We show, in Theorem~\ref{thm_Froissart_coprime}, how this coprimeness parameter allows one to monitor the absence of Froissart doublets, and so  generalizes and sharpens the previous attempts found in \cite{BM,BL}.
In \S 2.2  we introduce two new parameters based on the spherical derivative. We show, in Theorem~\ref{thm_spherical}(a),(b) and Corollary~\ref{cor_residuals}, that these parameters are essentially best Lipschitz constants of a rational function, and that these new parameters also allow one to insure the absence of Froissart doublets and poles with small residual.
In addition, in Theorem~\ref{thm_spherical2} we describe some special cases where these findings are sharper than those of Theorem~\ref{thm_Froissart_coprime}. Finally, in \S 2.3 we come back to the question of comparing distances in $\calR{m,n}$.
We show, in Theorem~\ref{thm_distances}, that \cite[Theorem~4.1]{BM} can be partly sharpened in terms of the coprimeness parameter and give an example showing that a second inequality cannot be improved. This completes the picture of Figure~\ref{figure}.

\begin{figure}
\centering{\includegraphics[scale=0.45]{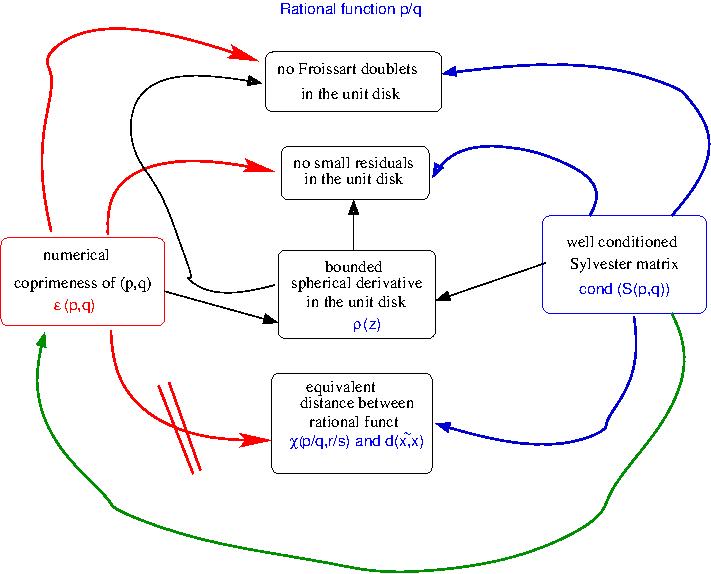}}
\caption{\em Link between the three indicators (Sylvester condition number, coprimeness, bounded spherical derivative) and Froissart doublets of rational functions.}\label{figure}
\end{figure}

The remainder of this paper is as follows.  Statements of our main results involving our three parameters are presented in three subsections in \S 2, with proofs of all our main statements given in \S 3. The paper then ends  with a conclusion and topics for future research in \S 4.

\section{Main results}

In this section we present the main results mentioned in the previous section.   Here all theorems are stated with the proofs  given later in the following section.

\subsection{Measure of coprimeness and Sylvester type matrices}

In what follows we consider fixed integers $m,n\geq 0$. In order to simplify notation, we will not explicitly indicate the dependency on $m,n$ of each object. For a polynomial $c(z)=c_0 + c_1z + \cdots +c_nz^n$ with coefficients $c_j$ we denote by
$\vect{c}=(c_0,c_1,\cdots ,c_n)^T$ its coefficient vector, with the size of this vector being clear from the context.
We start by introducing a so-called Sylvester type matrix $S^{(\ell)}$ associated to a pair of polynomials.

\begin{definition}\label{def_Sylvester}
   Given an integer $\ell\geq 0$ and polynomials $p\in\cc_m [z], q\in\cc_n [z]$, with coefficients $p_j,~q_j$, respectively, the associated $(m+n+\ell) \times (m + n + 2 \ell)$ Sylvester type matrix of $p$ and $q$ is defined by
$$
S^{(\ell)}(p,q)=\left(
\begin{array}{cc}
\underbrace{\begin{array}{cccc}
p_0 &&&\\
p_1 & p_0&&\\
\vdots &p_1&\ddots&\\
p_m &\vdots&\ddots&p_0\\
& p_m & &p_1\\
&&\ddots & \vdots \\
&&&p_m
\end{array} }_{n+\ell}&
\underbrace{\begin{array}{cccc}
q_0 &&&\\
q_1 & q_0&&\\
\vdots &q_1&\ddots&\\
q_n &\vdots&\ddots&p_0\\
& q_n & &q_1\\
&&\ddots & \vdots \\
&&&q_n
\end{array} }_{m+\ell}
\end{array}
\right)\in \cc^{(m+n+\ell)\times (m+n+2\ell)}.
$$
\qed
\end{definition}
\noindent
When $\ell = 0$,   $S^{(\ell)}(p,q)$  reduces to the transpose of the classical Sylvester matrix \cite{Labahn} while when $\ell = 1$ we get the Sylvester type matrix used in  \cite{BM}. The more general  $\ell$ allows us to consider increased degrees in an associated diophantine equation connected to polynomial gcd computation of $p$ and $q$.

It is well-known \cite{Labahn} that the classical square Sylvester matrix $S^{(0)}(p,q)$ is invertible if and only if the polynomials $p$ and $q$ are coprime and the defect $\min (m-\deg(p),n-\deg(q)) $ is equal to zero, that is, the rational function $p/q$ is nondegenerate. More generally, $S^{(\ell)}(p,q)$ has full row rank if and only $p/q$ is nondegenerate. We refer to \cite[Lemma~3.1]{BM} for a proof in the case $\ell=1$, while a proof for $\ell>1$ is similar, based on the relation
\begin{eqnarray} \label{S_powers_z}
    &&
    \Bigl(1,~z,...,~z^{m+n+\ell-1}\Bigr) S^{(\ell)}(p,q)
    \\&& \nonumber=
    \Bigl(p(z),~z^1p(z),~...,~z^{n+\ell-1}p(z),
    ~q(z),~z^1q(z),~...,~z^{m+\ell-1}q(z)\Bigr).
\end{eqnarray}

In order to make the link with the coprimeness parameter discussed by Corless, Gianni, Trager, and Watt in \cite{Cor} we introduce as in \cite{BM} a norm in $\mathbb C[z] \times \mathbb C[z]$ through the formula
$$
      \| (p,q) \|_2 = \sqrt{\| \vect{p} \|_2^2 + \| \vect{q} \|_2^2} ,
$$
and consider the following quantities.
\begin{definition}\label{def_epsi}
For $p\in\cc_m[z]$, $q\in\cc_n[z]$, and a set $K \subset \mathbb C$, consider
\begin{eqnarray*} &&
    \epsi{2}{}(p,q)=\inf \bigl\{ \| (p-\widetilde p,q-\widetilde q) \|_2 :
    (\widetilde p,\widetilde q)\in \mathbb C_m[z]\times \mathbb C_n[z]
    \mbox{~have a common root} \bigr\},
    \\&&
    \epsi{2}{K}(p,q)=\inf_{z\in K}
    \left(\frac{\mo{p(z)}^2}{\sum_{j=0}^m \mo{z}^{2j}}+\frac{\mo{q(z)}^2}{\sum_{j=0}^n \mo{z}^{2j}}\right)^{1/2}.
\end{eqnarray*}
\end{definition}

Therefore the coprimeness parameter $\epsi{2}{}(p,q)$ measures the distance to the set of pairs of polynomials with a non-trivial gcd. At the same time it is mentioned in \cite[Remark 4]{Cor}  that $\epsi{2}{}(p,q)$ coincides with $\epsi{2}{\mathbb C}(p,q)$, the latter quantity being much more accessible since one minimizes only with respect to the single complex parameter $z$.


Since $\| S^{(\ell)}(p,q)\|_2$ is not too far from $\| (p,q)\|_2$, our coprime measure $\epsi{2}{}(p,q)$
approximately also gives the distance of $S^{(\ell)}(p,q)$ to the set of singular Sylvester type matrices $S^{(\ell)}(\widetilde p,\widetilde q)$, that is, a kind of smallest structured singular value \cite{rump}. As such, an estimate of the form
\begin{equation} \label{eq_epsi_inverse}
     \epsi{2}{\mathbb C}(p,q) =\epsi{2}{}(p,q) \geq \frac{1}{\sqrt{m+n+1} \,  \| S^{(\ell)}(p,q)^\dagger\|_2}
\end{equation}
in terms of the norm of the pseudo-inverse is not surprising.
 To see this  we just take norms in \eqref{S_powers_z} (see also  \cite[Lemma~5.1]{BM} for a proof in the case $\ell=1$).
In \cite[\S 6.2]{BM} we conjectured that the dependency on $\ell$ of the right-hand side of \eqref{eq_epsi_inverse} is not important. In the present paper we are able to state:
%
\begin{theorem}\label{thm_norms}
   Let $p\in \mathbb C_m[z],q\in \mathbb C_n[z]$ such that
   $p/q$ is nondegenerate. Then
   \begin{equation}\label{relSl}
      \nor{S^{(0)}(p,q)^{-1}}_2 ~ \leq ~ \nor{S^{(\ell)}(p,q)^{\dagger}}_2 ~ \leq ~ (1+\sqrt{\ell})\nor{S^{(0)} (p,q)^{-1}}_2
   \end{equation}
   for all integers $\ell\geq 0$.
\end{theorem}


\begin{remark}\label{rem_norm1}
   The authors in \cite{BL} have obtained more compact expressions by choosing in Definition~\ref{def_epsi} a different norm for pairs of polynomials, namely
   $$
       \| (p,q) \|_1 = \max( \| \vect{p}\|_1, \| \vect{q}\|_1).
   $$
   This allowed them to deduce that $\| S^{(\ell)}(p,q) \|_1 = \| (p,q) \|_1$, independent of $\ell$. In this case, the one-norm equivalent of Definition~\ref{def_epsi} becomes
   \begin{eqnarray*}
         \epsi{1}{}(p,q) & :=&
    \inf \bigl\{ \| (p-\widetilde p,q-\widetilde q) \|_1 :
    (\widetilde p,\widetilde q)\in \mathbb C_m[z]\times \mathbb C_n[z]
    \mbox{~have a common root} \bigr\}
    \\&=& \inf \bigl\{ \| S^{(\ell)}(p,q) - \widetilde S \|_1 :
    \widetilde S
    \mbox{~a Sylvester type matrix of not full row rank} \bigr\}.
   \end{eqnarray*}
   It was also shown in \cite[Theorem~4.1]{BL} that $\epsi{1}{}(p,q)=\epsi{1}{\mathbb C}(p,q)$, where
   {\small
   \begin{eqnarray*}
       \epsi{1}{K}(p,q) & := & \inf_{z\in K}\max \conj{\frac{\mo{p(z)}}{\max\left(1,\mo{z}^m\right)},
\frac{\mo{q(z)}}{\max\left(1,\mo{z}^n\right)}} \nonumber \\
& = & \inf_{z\in K}\frac{\nor{(1,z,\cdots , z^{m+n+\ell-1})S^{(\ell )}(p,q)}_1}{\nor{(1,z,\cdots , z^{m+n+\ell-1})}_1}.
   \end{eqnarray*}
   }
The last relation implies $\epsi{1}{}(p,q)\geq \frac{1}{ \| S^{(0)}(p,q)^{-1}\|_1}$, as mentioned already in \cite[Lemma~2.1]{BL}. \qed
\end{remark}

Let us now turn to the question of existence of a Froissart doublet for a rational function $r=p/q\in \calR{m,n}$, that is, a pair consisting of a zero $z_p$ and a pole $z_q$ of $r$ 
which are close to each other.
In \cite[Theorem~1.3(a)]{BM} it was shown that
\begin{equation} \label{Froissart_BM}
     | z_p - z_q | \geq \frac{1}{3\sqrt{2}(m+n+1)^{3/2} \mbox{cond}(S^{(1)}(p,q))},
\end{equation}
provided that both $z_p$ and $z_q$ are in the closed unit disk $\mathbb D$. Here
$\mbox{cond}(B)=\| B \|_2 \, \| B^\dagger \|_2$ denotes the condition number with respect to the euclidean norm.

It seems reasonable to expect that a sufficiently large $\epsi{s}{}(p,q)$ also implies the absence of Froissart doublets, since in this case $p,~q$ is relatively far from a pair of polynomials having a non-trivial gcd. Let
$$
    \chi(x,y):=\frac{|x-y|}{\sqrt{1+|x|^2}\sqrt{1+|y|^2}}.
$$
be the chordal metric obtained by takng the euclidean distance on the Riemann sphere $\mathbb{S}^2$ which is identified with the extended complex plane $\C \cup \{\infty\}$ through stereographic projection. Then the distance between a pole and a zero of a rational function as measured using the chordal metric is approximated from below by:

%
\begin{theorem}\label{thm_Froissart_coprime}
 Let $K\subset \cc$ and $r=\frac{p}{q} \in \calR{m,n}$. Then for any pair $z_p,z_q\in \mathbb C$ with $p(z_p)=0,q(z_q )=0$ and $s \in \{1,2\}$ we have that
\begin{equation}\label{Froissart_BLM}
   \frac{1}{2}\frac{ \epsi{s}{K}(p,q)} {\max\left(~m\nor{\vect p}_s,~n\nor{\vect q}_s ~\right)}  \leq \chi (z_p,z_q).
\end{equation}
Moreover, if $K\subset\D$ or $1/K\subset \D$ we can replace the maximum in the denominator by a minimum.
\end{theorem}
Theorem \ref{thm_Froissart_coprime} thus implies that $p$ and $q$ numerically relatively prime (that is, having a large $\epsi{s}{K}(p,q)$) then implies that $r = p/q$ cannot have any Froissart doublets.
Note that,  combined with the estimate \eqref{eq_epsi_inverse}, we have  that the inequality in Theorem~\ref{thm_Froissart_coprime} is sharper than \eqref{Froissart_BM}. Special cases of Theorem~\ref{thm_Froissart_coprime} have been claimed without proof in \cite[\S 4]{BL} for $s=1$ and $K=\mathbb C$, and
established in \cite[Lemma~6.1]{BM} for $s=2$ and $K=\mathbb D$.

It is interesting to note  that the indicators used in \eqref{Froissart_BM} and \eqref{Froissart_BLM} are not sensitive with respect to a small perturbation of the numerator and denominator. Here we can state the following.

\begin{theorem}\label{thm_sensitivity}
    Let $K\subset \cc$ and $\frac{p}{q},\frac{\widetilde p}{\widetilde q} \in \calR{m,n}$.
    \\ {\bf (a)}
    If $\frac{p}{q}$ is nondegenerate and
    $$
               {\| (p-\widetilde p,q-\widetilde q)\|_2} \leq
               \frac{1}{3 \sqrt{m+n+1} \, \| S^{(1)}(p,q)^\dagger \|_2}
   $$
   then $\frac{1}{2} \leq \mbox{cond}(S^{(1)}(\widetilde p,\widetilde q))/\mbox{cond}(S^{(1)}(p,q)) \leq 2$.
    \\ {\bf (b)}
    Let $s\in \{ 1,2\}$. If
    $$
       \| (p-\widetilde p,q-\widetilde q)\|_s \leq \frac{1}{2}
       \epsi{s}{K}(p,q)
    $$
    then
    $\frac{1}{2} \leq \epsi{s}{K}(\widetilde p,\widetilde q)/\epsi{s}{K}(p,q) \leq 3/2$.
\end{theorem}
Notice that, according to \eqref{eq_epsi_inverse}, the neighborhood in part (b) for $s=2$ is larger than the neighborhood in part (a).

We are now able to show that inequalities \eqref{Froissart_BM} and \eqref{Froissart_BLM} are robust in the sense that they remain valid up to some modest constant if $z_p$ and $z_q$ are roots not of $(p,q)$ but of some perturbed $(\widetilde p,\widetilde q)\in \mathbb C_m[z]\times \mathbb C_n[z]$ sufficiently close to $(p,q)$. For \eqref{Froissart_BM} this has been done before in \cite[Theorem~1.3(a)]{BM}, and we essentially repeat their arguments. For \eqref{Froissart_BLM}, it is convenient to write first the slightly weaker inequality
$$
   \chi (z_p,z_q)  \geq \frac{ \epsi{s}{K}(p,q)} {2(m+n) \, \| (p,q) \|_s}  ,
$$
and to observe that $\epsi{s}{K}(p,q)\leq \| (p,q) \|_s$. Then Theorem~\ref{thm_sensitivity}
yields the following.

\begin{corollary}
    Let $K\subset \cc$ and $\frac{p}{q},\frac{\widetilde p}{\widetilde q} \in \calR{m,n}$. Furthermore, let $z_p,z_q\in \mathbb C$ with $\widetilde p(z_p )=\widetilde q(z_q )=0$. Then, under the assumptions of  Theorem~\ref{thm_sensitivity}(a),
    $$
          | z_p - z_q | \geq \frac{1}{6\sqrt{2}(m+n+1)^{3/2} \mbox{cond}(S^{(1)}(p,q))},
    $$
    Furthermore, under the assumptions of Theorem~\ref{thm_sensitivity}(b),
    $$
   \chi (z_p,z_q)  \geq
   \frac{ \epsi{s}{K}(p,q)} {6(m+n) \, \| (p,q) \|_s} .
    $$
\end{corollary}
Again, using \eqref{eq_epsi_inverse} one may show that the second statement for $s=2$ implies the first one.

\subsection{Froissart doublets, small residuals and spherical derivatives}

In this subsection we will introduce a new parameter in order to monitor the existence of Froissart doublets. Recall that the spherical derivative of a rational function $r\in \calR{m,n}$ is given by
\begin{equation} \label{def_spherical}
   \rho(r)(z) = \frac{|r'(z)|}{1+|r(z)|^2}
\end{equation}
while for any $K \subset \mathbb C$ we set
\begin{equation}\label{def_sphericalK}
\rho_K(r):=\sup_{z\in K} \rho(r)(z).
\end{equation}
Note that $\frac{1}{\rho(r)(z_q)}$ equals the modulus of the residual of a simple pole $z_q$. Hence the following statement,  complementing \cite[Theorem~1.3(b)]{BM},  is immediate.
\begin{corollary}\label{cor_residuals}
   Let $\beta$ be the residual of a simple pole $z_q$ of $r$ in $K$. Then
   $|\beta|\geq \frac{1}{\rho_K(r)}$.
\end{corollary}

We show below that the quantity $\rho_K(r)$ is the best Lipschitz constant for $r$ in $K$.
As such, for a reliable evaluation of $r(z)$ for $z\in K$, it seems to be reasonable to restrict ourselves to rational functions $r$ with modest $\rho_K(r)$.
On the other hand, if we want to measure the distance of arguments in terms of the chordal metric, another indicator is more appropriate, namely
\begin{equation} \label{def_spherical2}
   \nu(r)(z) = (1+|z|^2)\frac{|r'(z)|}{1+|r(z)|^2}
    \quad ~~
   \mbox{ and } ~~\nu_K(r):=\sup_{z\in K} \nu(r)(z).
\end{equation}

Let us now turn to Froissart doublets and compare our new indicators with those given previously.

\begin{theorem}\label{thm_spherical}
   Let $K\subset \cc$ and $r= \frac{p}{q}\in \calR{m,n}$ with $p$ and $q$ coprime and  $z_p,z_q\in \mathbb C$ with $p(z_p )=q(z_q )=0$.
   \begin{itemize}
   \item[{\bf (a)}]
   If $K$ is convex then
   \begin{equation} \label{Lipschitz}
       \rho_K(r)=\sup_{z_1,z_2\in K} \frac{\chi(r(z_1),r(z_2))}{|z_1-z_2|} .
   \end{equation}
   In particular, $|z_p-z_q| \geq \frac{1}{\rho_K(r)}$.
   \item[{\bf (b)}]
   If $K$ is spherically convex\footnote{This means that, with $z_1,z_2\in K$, also the preimage of the shortest path from $z_1$ to $z_2$ on the Riemann sphere belongs to $K$. Notice that disks and half-planes are spherically convex.} then
   \begin{equation} \label{Lipschitz2}
       \nu_K(r) \leq \sup_{z_1,z_2\in K} \frac{\chi(r(z_1),r(z_2))}{\chi(z_1,z_2)} \leq \frac{\pi}{2} \nu_K(r) .
   \end{equation}
   In particular, $\chi(z_p,z_q) \geq \frac{2}{\pi \nu_K(r)}$.
   \end{itemize}
\end{theorem}

It is also interesting to explore the links between the spherical derivative and the numerical measure of coprimeness.
Here the following observation is helpful.

\begin{remark}
Notice that, by definition
$$
\epsi{s}{K}( p^m, q^m ) = \epsi{s}{K}(p,q)^m
$$
strongly depends on $m$, whereas
$\rho_K (r^m)(z)\leq 2m \rho_K (r)(z)$ for any rational function $r$. This follows from
\begin{eqnarray*}
\rho (r^m) (z)&=& \frac{m\mo{r(z)}^{m-1} \mo{r'(z)}}{1+\mo{r(z)}^{2m}} \leq \frac{\mo{r(z)}^{m-1}}{\max (1,\mo{r(z)}^{2m-2})} \frac{m \mo{r'(z)}}{\max (1,\mo{r(z)}^{2})}\\
&\leq& m\frac{2\mo{r'(z)}}{1+\mo{r(z)}^2}=2m\rho (r)(z),
\end{eqnarray*}
using the fact that $2\max (1,\mo{r(z)}^2)\geq 1+\mo{r(z)}^2$ and then taking sup over $K$.
\qed
\end{remark}

In the following example, which was already studied in \cite[Example 5.3]{BL}, we see that the bounds of Theorem~\ref{thm_spherical}(a) is approximately sharp whereas Theorem~\ref{thm_Froissart_coprime} is not of the same order, at least for larger $m=n$.

\begin{example}\label{exampleBL}
   Consider $r=( \frac{p}{q} )^m$  for $p(z)=z$, $q(z)=\frac{z-1}{2}$ with $m\geq 0$ an integer. The poles and zeros of $r$ lie in the closed unit disk
   $K=\mathbb D$ (or $K=[0,1]$) and have a Euclidean distance $1$ or spherical distance $1/2$. However, $\| \vect{p^m} \|_1=\| \vect{q^m} \|_1 = 1$.
In addition we note without proof that
   $$
      \epsi{1}{}(p^m,q^m)=\epsi{1}{}(p,q)^m= 3^{-m}
~~\mbox{ and }
\rho_{K}(\frac{p}{q}) = \rho_{\{1/3\}}(\frac{p}{q}) = \frac{9}{4}.
$$
Since  $\rho_K(r)\leq 2 m \cdot \rho_K(\frac{p}{q})$ by our previous remark, we can then compare the spherical derivative by
   $$
       \rho_{K}(r) \leq 2 m \, \rho_{K}(\frac{p}{q}) = 2m \rho{\{1/3\}}( \frac{p}{q})
       = \frac{9m}{2}.
   $$
\qed
\end{example}

Inequalities comparing the spherical derivative and the numerical measure of coprimeness are given in the following.

\begin{theorem}\label{thm_spherical2}
   Let $K\subset \cc$ and $r= \frac{p}{q}\in \calR{m,n}.$  
   If $K \subset \mathbb D$ or $m=n$ then
   \begin{equation} \nonumber
      \frac{1}{2}\frac{ \epsi{1}{K}(p,q)} {\max\left(~m\nor{\vect p}_1,~n\nor{\vect q}_1~\right)}
      \leq \frac{1}{\nu_K(r)} \leq \frac{1}{\rho_K(r)}.
   \end{equation}
\end{theorem}
Theorem~\ref{thm_spherical2} identifies some particular cases where the bounds of Theorem~\ref{thm_spherical}(a),(b)  are sharper than the bound \eqref{Froissart_BLM} of Theorem~\ref{thm_Froissart_coprime}. In addition, in the case of a simple pole $\widetilde{z}$ of $r=p/q$,  Theorem \ref{thm_spherical2} combined with Corollary \ref{cor_residuals} also implies that $p$ and $q$ numerically relatively prime implies no small residual at $\widetilde{z}$.

\subsection{Distance of rational functions}

Numerical analysis in $\calR{m,n}$ requires one to measure distances between rational functions $r=\frac{p}{q} \in\calR{m,n}$ and $\widetilde r= \frac{\widetilde p}{\widetilde q} \in \calR{m,n}$. As mentioned in \cite{BM} several choices are possible. If one is interested in values, then the choice
$$
    \chi_K(r,\widetilde r) := \sup_{z\in K} \chi(r(z),\widetilde r(z))
$$
could be the most suitable since the chordal metric measures the euclidean distance of points in the Riemann sphere. On the other hand, if one prefers to define a distance in terms of the coefficients of numerators and denominators, then one should take care of the fact that coefficient vectors are only unique up to a scaling with a complex factor, that is, the norm and the phase. In \cite{BM} the authors made the choice
$$
    d(r,\widetilde r) = \min \left\{ \left\|
    \frac{1}{\|(p,q)\|_2}\left[\begin{array}{cc} \vect{p} \\ \vect{q} \end{array}\right]
    -
    \frac{a}{\|(\widetilde p,\widetilde q)\|_2}\left[\begin{array}{cc} \vect{\widetilde p} \\ \vect{\widetilde q} \end{array}\right] \right\| : a \in \mathbb C, |a|=1 \right\},
$$
and hence
$$
    d(r,\widetilde r) = \min \left\{ \left\|
    \frac{1}{\|(p,q)\|_2}\left[\begin{array}{cc} \vect{p} \\ \vect{q} \end{array}\right]
    \pm
    \frac{1}{\|(\widetilde p,\widetilde q)\|_2}\left[\begin{array}{cc} \vect{\widetilde p} \\ \vect{\widetilde q} \end{array}\right] \right\|  \right\}
$$
in the case $r,\widetilde r \in \mathbb R_{m,n}(z)$ of real coefficients.
In \cite[Theorem~4.1]{BM} it is shown that when $r$ is nondegenerate we have
\begin{equation} \label{distances}
       \frac{(m+n+1)^{-3/2}}{\sqrt{2}\,\mbox{cond}(S^{(1)}(p,q))}
       \leq \frac{\chi_{\mathbb D}(r,\widetilde r)}{d(r,\widetilde r)}
       \leq \sqrt{2(m+n+1)}\, \mbox{cond}(S^{(1)}(p,q)).
\end{equation}
Thus, roughly speaking, the two distances are comparable provided that $\mbox{cond}(S^{(1)}(p,q))$ is modest, or, in other words, the rational function $r$ is well-conditioned.

In view of the preceding statements, it is then  natural to wonder whether similar inequalities are kept if we replace $\mbox{cond}(S^{(1)}(p,q))$ in \eqref{distances} by $\frac{1}{\epsi{1}{K}(p,q)}$. The following theorem shows that this is possible for the right-hand side of \eqref{distances}.
\begin{theorem}\label{thm_distances}
If $K \subset \mathbb D$ or $m=n$ then  for any $r = \frac{p}{q}, ~\widetilde{r}=\frac{\widetilde{p}}{\widetilde{q}} \in \calR{m,n}$ we have
   $$
       \epsi{1}{K}(p,q) \, \cdot \chi_K( r, {\widetilde r})\leq \sqrt{2}\nor{(p-\widetilde p,q-\widetilde q)}_1.
   $$
\end{theorem}

On the other hand the corresponding statement does not hold for the left-hand side of \eqref{distances}  as long as we want to keep constants which are only polynomially growing in $m+n$. That is, there is not a quantity $C(m+n)$ of order a small power of $m+n$ for which
\begin{equation}\label{exem}
\chi_{\D}\left( r, \tilde{r} \right)\geq C(m+n)\nor{(p-\tilde{p},q-\tilde{q})}_1\epsi{1}{\mathbb D}(p,q),
\end{equation}
Indeed, in the following example we let $m = n$ and show that for each $m$ there are  polynomials $p_m,q_m,\widetilde p_m,\widetilde q_m\in \mathbb R_m[z]$, with corresponding rational functions $r_m$ and $\widetilde{r}_m$,  such that, up to some constants,
\begin{equation}\label{eqn1}
       \frac{
       \chi_{\mathbb D}(r_m, \widetilde{ r}_m)     }{
       \| (p_m- \widetilde p_m,q_m-\widetilde q_m)\|_1
       }
       \sim \frac{\sqrt{m}(\frac{3}{8})^m}{\epsi{1}{\mathbb D}(p_m,q_m)}
       \sim \frac{9^m}{\mbox{cond}(S^{(1)}(p,q))}
       \sim \sqrt{m}\left(\frac{27}{8}\right)^m \, \epsi{1}{\mathbb D}(p_m,q_m),
 \end{equation}
are all of the same order of magnitude but that
 \begin{equation}\label{eqn2}
            \frac{\chi_{\mathbb D}(r_m, \widetilde {r}_m)}
        {d(r_m, \widetilde {r}_m)} \, \frac{1}{\epsi{1}{\mathbb D}(p_m,q_m)}
\end{equation}
grows at least as a constant times $\sqrt{m}(\frac{27}{8})^m$.

\begin{example}
    Let
   $
      p_m(z)=z^m , \quad q_m(z)= \left(\frac{z-1}{2}\right)^m
   $
be the polynomials from Example~\ref{exampleBL},
and suppose these are perturbed as
$$
      (\widetilde p_m,\widetilde q_m)=(p_m-\eta u_m, q_m-\eta v_m)
 $$
   where $u_m,v_m\in \mathbb R_{m-1}[z]$ are such that $q_m(z)u_m(z) - p_m(z)v_m(z) = 1$, and $\eta$ is a small parameter which we will fix  later.
   Then
   \begin{equation} \label{norm1}
      \| S^{(1)}(p_m,q_m) \|_1 = \| \vect{p_m} \|_1 = \| \vect{q_m} \|_1
      = \| (p_m,q_m)\|_1 = 1 .
   \end{equation}
   In
   \cite[Examples 5.2 and 5.3]{BL} 
explicit formulas for
   $u_m,v_m$ and $S^{(0)}(p_m,q_m)^{-1}$ were derived, allowing the authors  to deduce that
   $$
         \| \vect{v_m} \|_1 \leq \| \vect{u_m} \|_1 \sim \frac{2^{3m-1}}{\sqrt{\pi m}} \quad  \mbox{ and } \quad
         \frac{ \| S^{(0)}(p_m,q_m)^{-1} \|_1 }{\| \vect{u_m}\|_1 } \in [1,2].
   $$
   As $\eta$ is still not fixed, we can now specify  that
\begin{equation} \label{assumption}
    2 \, \| (p-\widetilde p,q-\widetilde q) \|_1  = 2 \, |\eta| \, \| \vect{u_m} \|_1 =  \epsi{1}{\mathbb D}(p_m,q_m),
   \end{equation}
 from which, using Example~\ref{exampleBL} and Theorem~\ref{thm_sensitivity}(b), we then have for all $z\in \mathbb D$,
   \begin{eqnarray*} &&
      \epsi{1}{}(p_m,q_m)=\epsi{1}{\mathbb D}(p_m,q_m)=3^{-m} \quad \mbox{ with }  \quad
     \frac{  \epsi{1}{\{ z \}}(\widetilde p_m,\widetilde q_m)}{\epsi{1}{\{ z \}}(p_m,q_m)}
      \in \left[\frac{1}{2}, \frac{3}{2} \right] .
   \end{eqnarray*}
   Since $(p_m \tilde{q}_m - \tilde{p}_m q_m)(z) = \eta$, we have
   \begin{eqnarray*}
       \chi_{\mathbb D}(r_m, \widetilde{r}_m)
       & = & \sup_{z\in \mathbb D} \frac{|\eta|}{\sqrt{(|p_m(z)|^2+|q_m(z)|^2)(|\widetilde p_m(z)|^2+|\widetilde q_m(z)|^2)}}
       \nonumber \\ &
       \geq & \sup_{z\in \mathbb D} \frac{|\eta|}{2 ~ \epsi{1}{\{ z \}}(\widetilde p_m,\widetilde q_m) \, \epsi{1}{\{ z \}}(p_m,q_m)}
       = \frac{|\eta|}{3 ~ \epsi{1}{\mathbb D}(p_m,q_m)^2}.
   \end{eqnarray*}
   Similarly
   \begin{eqnarray*} &&
       \chi_{\mathbb D}(r_m, \widetilde{r}_m)
       \leq \sup_{z\in \mathbb D} \frac{|\eta|}{\epsi{1}{\{ z \}}(\widetilde p_m,\widetilde q_m) \, \epsi{1}{\{ z \}}(p_m,q_m)} \leq
       \frac{2 |\eta|}{\epsi{1}{\mathbb D}(p_m,q_m)^2} .
   \end{eqnarray*}
   Thus, up to some constants,
   $$
       \frac{
       \chi_{\mathbb D}(r_m, \widetilde{r}_m)     }{
       \| (p_m- \widetilde p_m,q_m-\widetilde q_m)\|_1
       } \sim \frac{1}{\epsi{1}{\mathbb D}(p_m,q_m)^2 \, \| \vect{u_m} \|_1},
   $$
which shows our perturbation satisfies equation \eqref{eqn1}.

   Finally, notice that \eqref{assumption} and \eqref{norm1} also imply that
   $$\| (\widetilde p,\widetilde q) \|_1 \in [1-|\eta|\, \| \vect{u_m} \|_1,1+|\eta| \, \| \vect{u_m} \|_1] \subset [\frac{1}{2}, \frac{3}{2}],$$ and thus
   $$
        d(r_m, \widetilde{r}_m) \leq \frac{2 |\eta| \, \| \vect{u_m} \|_1}{1-|\eta|\, \| \vect{u_m} \|_1}
        \leq 4\, |\eta|\, \| \vect{u_m} \|_1.
   $$
   Combining, we find that
   $$
   \frac{\chi_{\mathbb D}(r_m, \widetilde{r}_m)}
        {d(r_m, \widetilde{r}_m)}
        \geq \frac{1}{12~ \|  \vect{u_m} \|_1 \, \epsi{1}{\mathbb D}(p_m,q_m)^2} ,
   $$
   from which the remaining growth statement \eqref{eqn2} follows. \qed
\end{example}

\section{Proofs of Theorems}

\begin{proof}[{Proof of Theorem~\ref{thm_norms}.} ]
Denote $S=S^{(0)} (p,q)$ and $S_*$, the matrix obtained by a permutation $P$ of columns of $S^{(\ell)}(p,q)$ such that
$$S_*= S^{(\ell)}(p,q) \cdot P=
\left( \begin{array}{c|c} S&E\\\hline 0 &V\end{array} \right),
$$
where $E\in\cc^{(n+m )\times \ell}$ and $V\in\cc^{(2\ell)\times \ell}$. From non degeneracy we know that $S$ is regular and that $S^{(\ell )}(p,q)$ is of maximum rank $n+m+\ell$.
Its kernel is therefore of dimension~$\ell$. Let $Y\in\cc^{(n+m+2\ell)\times\ell}$ be a matrix whose columns generate  Ker$(S_*)$. We know that the orthogonal projector in Ker$(S_*)$ is $Y(Y^*Y)^{-1}Y^*=YY^{\dagger}$ and so,
 as $S_*^{\dagger}S_*$ is the projector in Ker$(S_*)^{\perp}$ we get
 $$S_*^{\dagger}S_*=I-YY^{\dagger}.$$
Let $B_2\in\cc^{(n+m+2\ell)\times\ell}$ be a matrix such that
$$S^{(\ell )}(p,q) \cdot B_2=\left(\begin{array}{c}0\\I_{\ell}\end{array}\right),\quad I_{\ell} \mbox{ identity of order }\ell.$$
We can construct the columns of $B_2$ in the following way.
Let $w$ be the last column of $S^{-1}$, that is, $Sw=e_{n+m}$. In polynomial language, $w$ contains the coefficients of two polynomials $ u$ and $v$ of degree $n-1$ and $m-1$, respectively, and satisfying the Bezout equation $$p(z)u(z)+q(z)v(z) = z^{m+n-1}.$$ The columns of $B_2$ contain the coefficients of $z^i \cdot u(z) $ and $z^i \cdot v(z)$, $i=1,\cdots \ell$. More precisely, the columns have the form
$$(\underbrace{0,\cdots, 0}_{i-1},\vect u^T,\underbrace{0,\cdots ,0}_{\ell -1},\vect v^T,\underbrace{0,\cdots , 0}_{\ell -i})\quad i=1,\cdots \ell .$$
We set $B_1=P^{-1}  \cdot B_2$ and so
 \begin{equation}\label{BB}
 B=\left(\begin{array}{c|c}
 \begin{array}{c}
 S^{-1}\\0\end{array} &B_1
 \end{array}  \right)
 \end{equation}
is a right inverse of $S_*$ since  $S_*B=I$. Then
\begin{equation}\label{SP} S_*^{\dagger}=\left(I-YY^{\dagger}\right) B\quad \mbox{ and } \quad \nor{S_*^{\dagger}}_2 ~ \leq ~ \nor{B}_2.
\end{equation}

 Let us now bound $\nor{B}_2$. From (\ref{BB}) and noting the fact that  
$\nor{B_1}_2$ is bounded by its Frobenius norm of $B_1$ and remembering that the columns of $B_1$ are constructed from  $Sw=e_{n+m}$ we get

 $$\nor{B}_2 ~\leq ~\nor{S^{-1}}_2+\nor{B_1}_2 ~\leq~ \nor{S^{-1}}_2 +\sqrt{\ell}\nor{w}_2 ~\leq~ (1+\sqrt{\ell})\nor{S^{-1}}_2.$$
  Thus  the second inequality follows.

 Let us now prove the first inequality of (\ref{relSl}). Here we make use of the formula
$$
\left( \begin{array}{cc} S&E\\ 0 &V\end{array} \right)^{\dagger}=
\left(\begin{array}{cc} S^{-1} & -S^{-1} EV^{\dagger}\\ 0 & V^{\dagger}\end{array}\right) .
  $$
for the pseudo-inverse of $S_*$. To see this consider the QR decomposition of $S$ and $V$
 $$S=Q_SR_S,\quad V=Q_V\left(\begin{array}{c}R_V\\0\end{array}\right)=\widetilde Q_V R_V,$$
 where $Q_S,Q_V$ are unitary matrices and $R_S$ and $R_V$ invertible since both $S$ and $V$ have maximum rank. Then $S_*$ can be written as the product
 $$S_*=\left(\begin{array}{cc}
 Q_S & 0\\
 0 & \widetilde Q_V\end{array}\right)
 \left(\begin{array}{cc}
 R_S & \widetilde E\\
 0 &  R_V\end{array}\right), \quad \widetilde E=Q_S^{-1}E
  $$
where the first matrix in the product has linearly independent columns and the second one has linearly independent rows. We then have the pseudo-inverse of $S_*$ as
$$
S_*^{\dagger}=
\left(\begin{array}{cc}
 R_S^{-1} & -R_S^{-1}\widetilde E R_V^{-1} \\
 0 & R_V^{-1}\end{array}\right)
 \left(\begin{array}{cc}
 Q^*_S & 0\\
 0 & \widetilde Q_V\end{array}\right)=\left(\begin{array}{cc} S^{-1} & - S^{-1} EV^{\dagger}\\ 0 & V^{\dagger}\end{array}\right) .
  $$
  This trivially gives $\nor{S^{-1}}_2 ~\leq ~ \nor{S_*^{\dagger}}_2$ and the result follows.
\end{proof}

\begin{proof}[{Proof of Theorem~\ref{thm_Froissart_coprime}.}]
Assume first that $p/q$ is nondegenerate.
If we let $\comp{p}, \comp{q}$ denote the polynomials with reversed coefficients of $p$ and $q$, respectively, then  we  get  $\epsi{s}{K} (p,q)=\epsi{s}{1/K}(\comp{p},\comp{q})$. Without loss of generality we may thus suppose $\min (\mo{z_p},\mo{z_q})\leq 1$.
We can write
\begin{equation}\label{zeros}\mo{p(z_q)}=\mo{p(z_q )-p(z_p)}\leq \sum_{k=1}^m \mo{p_k}\mo{z_q^k-z_p^k}=\mo{z_q-z_p}\sum_{k=1}^m\mo{p_k}\mo{\sum_{j=0}^{k-1}z_q^jz_p^{k-j-1}} .
\end{equation}
Let us suppose $\mo{z_q}\geq \mo{z_p}$ and so $\mo{z_p}\leq 1$.
 Then by twice applying  the Cauchy-Schwarz inequality we obtain
$$
\mo{p(z_q)}\leq \mo{z_q-z_p} \left\{\begin{array}{l}
\nor{\vect p}_2 m \cdot \left(\sum_{j=0}^{m-1}\mo{z_q}^{2j}\right)^{1/2}\\
\nor{\vect p}_1 m \cdot \max (1,\mo{z_q}^{m-1})
\end{array}\right. .$$
The definition of the chordal metric implies
$\mo{z_p-z_q}\leq \sqrt{2} ~\chi (z_p,z_q)~\sqrt{1+\mo{z_q}^2}$, and using
$$
\sqrt{( 1 + \mo{z_q}^2)\left( \sum_{j=0}^{m-1} \mo{z_q}^{2j} \right)} ~= ~\sqrt{\sum_{j=0}^{m-1} \mo{z_q}^{2j} + \sum_{j=1}^{m} \mo{z_q}^{2j} } ~\leq~ \sqrt{2} \sqrt{\sum_{j=0}^{m} \mo{z_q}^{2j}}
$$
we obtain
$$\mo{p(z_q)}\leq 2 ~\chi (z_p,z_q)\times\left\{\begin{array}{l}
\nor{\vect p}_2 m \cdot \left(\sum_{j=0}^{m}\mo{z_q}^{2j}\right)^{1/2}\\
\nor{\vect p}_1 m \cdot \max (1,\mo{z_q}^{m}) .
\end{array}\right.
$$
Thus  by the definition of $\epsi{s}{K}(p,q)$
$$\epsi{s}{K} (p,q)\leq 2 ~ \chi (z_p,z_q)~ m\nor{\vect p}_s .
$$
Similarly, if $\mo{z_p}\geq \mo{z_q}$ then we get
$$\epsi{s}{K} (p,q) \leq 2 ~ \chi (z_p,z_q)~n\nor{\vect q}_s$$
and the result then follows.

Finally we remark that the result follows trivially if $p/q$ is degenerate.
\end{proof}

\begin{proof}[Proof of Theorem~\ref{thm_sensitivity}]
    For part {\bf (a)} we use a Neumann series argument similar to the proof of \cite[Lemma~5.1]{BM}. Define
    $$
          E =  S^{(1)}(p,q)^{\dagger}  \Bigl(  S^{(1)}(p,q) - S^{(1)}(\widetilde p,\widetilde q) \Bigr) .
    $$
    Then by assumption and by comparison with the Frobenius norm $\| \cdot \|_F$ we have
    \begin{eqnarray*}
         \| E \|_2 &\leq&
         \| S^{(1)}(p-\widetilde p,q-\widetilde q) \|_F \, \| S^{(1)}(p,q)^{\dagger} \|_2
         \\&\leq& \sqrt{m+n+1} \,
         \| (p-\widetilde p,q-\widetilde q) \|_2 \, \| S^{(1)}(p,q)^{\dagger} \|_2
         \leq 1/3 .
    \end{eqnarray*}
    Since also by assumption $S^{(1)}(p,q)$ has full row rank, we find that
    $$
        S^{(1)}(\widetilde p,\widetilde q) =
        S^{(1)}(p,q) (I-E)
    $$
    and thus $(I-E)^{-1} S^{(1)}(p,q)^\dagger$ is a right inverse of
    $S^{(1)}(\widetilde p,\widetilde q)$. This implies that
    \begin{eqnarray*}
         \mbox{cond}(S^{(1)}(\widetilde p,\widetilde q))
         &\leq& \| S^{(1)}(\widetilde p,\widetilde q) \|_2 \,
         \| (I-E)^{-1} S^{(1)}(p,q)^\dagger \|_2
         \\&\leq &\frac{1 + \| E \|_2}{1-\| E \|_2}
         \mbox{cond}(S^{(1)}(p,q))
         \leq ~ 2 ~
         \mbox{cond}(S^{(1)}(p,q)).
    \end{eqnarray*}
    The other inequality in part {\bf (a)} follows by symmetry.

    For part {\bf (b)}, we first notice that, for any $z\in K$,
    \begin{eqnarray*} &&
         | \epsi{s}{\{ z \}}(\widetilde p,\widetilde q) - \epsi{s}{\{ z \}}(\widetilde p,\widetilde q) | =
         \\ &&
         \Bigl| \, \Bigl\|
          [  \frac{p(z)}{\| (1,z,...,z^m)\|_s}
         \frac{q(z)}{\| (1,z,...,z^n)\|_s} ] \Bigr\|_s
         -
          \Bigl\| [  \frac{\widetilde p(z)}{\| (1,z,...,z^m)\|_s}
         \frac{\widetilde q(z)}{\| (1,z,...,z^n)\|_s} ] \Bigr\|_s \Bigr|
         \\ &&
         \leq \Bigl\|
          [  \frac{p(z)-\widetilde p(z)}{\| (1,z,...,z^m)\|_s}
         \frac{q(z)-\widetilde q(z)}{\| (1,z,...,z^n)\|_s} ] \Bigr\|_s
         \leq \Bigl\|
          [  \| \vect{p-\widetilde p} \|_s , \| \vect{q-\widetilde q}, \|_s  ] \Bigr\|_s
          \\&& = \| (p-\widetilde p,q-\widetilde q) \|_s
          \leq \frac{\epsi{s}{K}(p,q)}{2} \leq \epsi{s}{\{ z \}}(p,q)
    \end{eqnarray*}
    where in the second last inequality we used our hypothesis.
    Thus
    $$
       \frac{1}{2}\epsi{s}{\{z\}}(p,q) \leq \epsi{s}{\{z\}}(\widetilde p,\widetilde q) \leq \frac{3}{2} \epsi{s}{\{z\}}(p,q)
    $$
    for all $z\in K$. The claim follows by taking the infimum for $z\in K$.
\end{proof}

\begin{proof}[{Proof of Theorem~\ref{thm_spherical}(a),(b).}]
 We start by observing that
 $$
 \lim_{x\to z}\frac{\chi (r(z), r(x))}{\mo{z-x}}= \rho (r) (z), \quad
 \lim_{x\to z}\frac{\chi (r(z), r(x))}{\chi (z,x)}= \nu (r) (z) ,
 $$
 and thus the suprema in \eqref{Lipschitz}, \eqref{Lipschitz2} are bigger than or equal to
 $\rho_K(r)$, and $\nu_K(r)$, respectively.
 Following \cite{Sch}, the spherical metric $\sigma(z_1,z_2)$ is given by
 the length of the shortest path on the Riemann sphere joining $z_1$ and $z_2$, and thus $$
    \sigma (z_1,z_2) = \min_{\Gamma}\int_{\Gamma}\frac{\mo{dz}}{1+\mo{z}^2}
 $$
 where $\Gamma$ is any differentiable curve joining $z_1$ to $z_2$, with the minimum being obtained for $\Gamma$ being the preimage of the shorter of the two circular arcs of radius $1/2$ on the Riemann sphere joining $z_1$ and $z_2$.
 By elementary trigonometry, we can link the chordal metric $\chi$ to the spherical metric via
 \begin{equation} \label{res1}
        \chi (z_1,z_2)\leq \sigma (z_1,z_2) \leq \frac{\pi}{2} \chi (z_1,z_2).
 \end{equation}
 Thus
 \begin{eqnarray}
     \chi (r(z_1),r(z_2))&\leq& \min_{\Gamma}\int_{r(\Gamma )}\frac{\mo{dw}}{1+\mo{w}^2}=\min_{\Gamma }\int_{\Gamma }\frac{\mo{r'(z)}}{1+\mo{r(z)}^2} (1+\mo{z}^2)\frac{\mo{dz}}{1+\mo{z}^2}\nonumber\\  && \nonumber \\
    &\leq &\left\{\begin{array}{l}
    \rho_K(r)\mo{z_1-z_2}\\
     \nu_K(r)\, \sigma (z_1,z_2)
    \end{array}\right. . \nonumber
 \end{eqnarray}
 Combined with \eqref{res1}, we obtain the remaining inequalities for establishing \eqref{Lipschitz}, \eqref{Lipschitz2}.

 For estimating the distance between pole and zero, notice that
 $$
     1=\chi(0,\infty) = \chi(r(z_p),r(z_q)) \leq \rho_K(r) \, | z_p - z_q|,
 $$
 and similarly for $\nu_K(r)$.
\end{proof}

\begin{proof}[{Proof of Theorem~\ref{thm_spherical2}.}]
   The inequality $\rho_K(r)\leq \nu_K(r)$ is an immediate consequence of the definition.
   Let $\nu_K(r)=\nu(r)(\widetilde z)$ for some $\widetilde z\in K$. We consider first the case $K \subset \mathbb D$ and thus $|\widetilde z|\leq 1$.
   Then $\nu_K(r)\leq 2 \rho(r)(\widetilde z)$, where
\begin{eqnarray*}
\rho (r)(\widetilde{z}) &=&\frac{\mo{p'(\widetilde{z})q(\widetilde{z})-p(\widetilde{z})q'(\widetilde{z})}}{\mo{p(\widetilde{z})}^2+\mo{q(\widetilde{z})}^2} \nonumber \\
   &\leq&  \frac{(\mo{p'(\widetilde{z})}+\mo{q'(\widetilde{z})})\max (\mo{p(\widetilde{z})},\mo{q(\widetilde{z})})}{\max (\mo{p(\widetilde{z})}^2,\mo{q(\widetilde{z})}^2 )}\\
&\leq &\frac{m \nor{p'}_1\max (1,\mo{\widetilde z}^m) + n \nor{q'}_1\max (1,\mo{\widetilde z}^n)}{\max (\mo{p(\widetilde z)},\mo{q(\widetilde z)})}\\
   &\leq &\frac{m\nor{p}_1}{\max (\mo{p(\widetilde z)},\mo{q(\widetilde z)})} + \frac{n\nor{q}_1}{\max (\mo{p(\widetilde z)},\mo{q(\widetilde z)})}.
 \end{eqnarray*}
 Thus when $|\widetilde z|\leq 1$, the assertion of Theorem~\ref{thm_spherical2} follows.

 If $|\widetilde z|>1$ and thus $m=n$ by hypothesis, we consider as before the reversed polynomials $$\overline{p}(z)=z^m p(1/z)\in \mathbb C_m[z],
 ~ \overline{q}(z)=z^n q(1/z)\in \mathbb C_n[z],$$ and observe that
 $\overline r(z):=\overline{p}(z)/\overline{q}(z)=r(1/z)$.
 Thus $$\nu_K(r)=\nu_{1/K}(\overline r)=\nu(\overline r)(1/\widetilde z)\leq 2 \rho(\overline r)(1/\widetilde z),$$ and the assertion follows making use of
 $\epsi{1}{K}(p,q)=\epsi{1}{1/K}(\overline p,\overline q)$.
\end{proof}

\begin{proof}[{Proof of Theorem~\ref{thm_distances}.}]
 Using the Cauchy-Schwarz inequality we get
\begin{eqnarray*}
\sqrt{\mo{p(z)}^2 + \mo{q(z)}^2 } \cdot \chi( r(z), \widetilde{r}(z) ) & = &
\frac{\mo{((p-\widetilde p)\widetilde q- (q-\widetilde q)\widetilde p)(z)}}{\sqrt{\mo{\widetilde p(z)}^2 +\mo{\widetilde q(z)^2}}} \\
&\leq & \sqrt{\mo{(p-\widetilde p)(z)}^2 + \mo{(q-\widetilde q)(z)}^2}\\
&\leq & \sqrt{\nor{p-\widetilde p}_1^2 \max (1,\mo{z}^{m})^2 ~+ \nor{q-\widetilde q}_1^2 \max (1,\mo{z}^{n})^2}\\
&\leq & \sqrt{2}\nor{(p-\widetilde p, q-\widetilde q)}_1\times\left\{
\begin{array}{ll}
1 & \mbox{ if } \mo{z}\leq 1\\
 \mo{z}^{n} &  \mbox{ if } \mo{z}\geq 1\mbox{ and } m=n .
\end{array} \right.
\end{eqnarray*}
We also have for $\mo{z}\leq 1$ that
\begin{eqnarray*}
\sqrt{\mo{p(z)}^2 + \mo{q(z)}^2 }&\geq & \max (\mo{p(z)},\mo{q(z)}) = \nor{(1,z, \cdots , z^{m+n})S^{(1)}(p,q)}_1\\
&\geq & \epsi{1}{K}(p,q)\nor{(1, z, \cdots  , z^{m+n})}_1 = \epsi{1}{K}(p,q).
\end{eqnarray*}
Using the definition of $\chi_K( r,~\widetilde{r})$ the result follows.

For $\mo{z}\geq 1$ and $m=n$  we get
\begin{eqnarray*}
\sqrt{\mo{p(z)}^2 + \mo{q(z)}^2 }&\geq &\frac{1}{\mo{z^{n}}}\max (\mo{z^{n}p(z)},\mo{z^{n}q(z)}\\
&=& \frac{1}{\mo{z^{n}}}\nor{(1,z,\cdots , z^{m+n})S^{(1)}(p,q)}_1 \\
&\geq &\frac{1}{\mo{z^{n}}}\epsi{1}{K}(p,q)\mo{z^{n+n}}
=\epsi{1}{K}(p,q)\mo{z^n} .
\end{eqnarray*}
Then
$$\chi(r(z),\widetilde {r}(z)) ~ \leq~ \sqrt{2}\nor{(p-\widetilde p, q-\widetilde q)}_1\frac{\mo{z}^{n}}{\epsi{1}{K}(p,q)\mo{z}^n}=\sqrt{2} \frac{\nor{(p-\widetilde p, q-\widetilde q)}_1}{\epsi{1}{K}(p,q)}$$
and the result follows.
\end{proof}

\section{Conclusions and topics for future research}

In this paper we have considered the problem of working with rational functions in a numeric environment, with the particular goal of monitoring the absence  of Froissart doublets.  Three different parameters were studied, including the euclidean condition number of underlying Sylvester-type matrices depending on some integer $\ell$, a parameter for determing coprimeness of two numerical polynomials and bounds on the spherical derivative.  Each case these parameters  sharpen those found in \cite{BM}.

Future plans include using our three parameters as penalties for computing rational approximants with the goal of removing such undesirable features. In addition, there are a number of open questions that fall out of our work.  All our results are given using a monomial  basis. It is of interest to see what can be said in the case of other polynomial bases, for example those based on orthogonal polynomials. It is also natural to look for
an interpretation of the quantity  $\nu (r)$ in terms of residues.


\begin{thebibliography}{99}
\bibitem{G-M} G. Baker and P. Graves-Morris: Pad\'e Approximants, Encyclopedia of Mathematics, Cambridge University Press (1996).
\bibitem{BGL} B. Beckermann, G. Golub and G. Labahn, On the Numerical Condition of a Generalized Hankel Eigenvalue Problem, {\em Numerische Mathematik} {\bf 106} (2007) 41-68.
\bibitem{BL} B. Beckermann and G. Labahn, When are two numerical polynomials relatively prime? {\em Journal of Symbolic Computation} {\bf 26} (1998) 677-689.
\bibitem{BM} B. Beckermann and A. Matos, Algebraic properties of robust Pad\'e approximants, {\em Journal of Approximation Theory }{\bf 190} (2015) 91-115.
\bibitem{Bessis} D. Bessis, Pad\'e approximations in noise filtering, J. Computational and  Applied Math.\, {\bf 66} (1996) 85-88.
\bibitem{Cor} R.M. Corless, P.M. Gianni, B.M. Trager and S.M. Watt, The singular value decomposition for polynomial systems, {\em Proceedings ISSAC'95}, Montreal, ACM Press (1995) 195--207.
\bibitem{cuyt} A. Cuyt and W-s Lee, Sparse interpolation and Rational approximation, In : {\em Contemporary Mathematics}, editors : Hardin, Lubinsky and Simanek. AMS (2016) 229-242.
\bibitem{driver} K. Driver, H. Prodinger, C. Schneider and J.A.C. Weideman, Pad\'e approximations to the Logarithm II: Identities, Recurrence and Symbolic Computation,
{\em Ramanujan Journal} {\bf 11} (2006) 139-158.
\bibitem{Froissart} M. Froissart, Approximation de Pad\'e: application \`a la physique des particules \'el\'ementaires,
    in RCP, Programme No. 25, v. 9, CNRS, Strasbourg (1969) 1-13.
\bibitem{gerhard} J. von zur Gathen and J. Gerhard, Modern Computer Algebra, Cambridge University Press, $3^{rd}$ edition (2013).
\bibitem{Labahn} K.O. Geddes, S.R. Czapor and G. Labahn, Algorithms for Computer Algebra, Springer Science \& Business Media (1992).
\bibitem{GLL} M. Giesbrecht, G. Labahn and W-s Lee,  Symbolic-numeric Sparse Interpolation of Multivariate Polynomials, {\em Journal of Symbolic Computation}, {\bf 44} (2009) 943-959.
\bibitem{GGT12} P. Gonnet, S. G\"uttel and L. N. Trefethen, Robust Pad\'e approximation via SVD, SIAM Review {\bf 55} (2013) 101-117.
\bibitem{KY07} E. Kaltofen and Z. Yang, On exact and approximate interpolation of sparse rational functions. {\em Proceedings of ISSAC 2007}, Waterloo, ACM Press (2007) 203-210.
\bibitem{masc} W.F.\ Mascarenhas, Robust Pad\'e Approximants may have spurious poles,
{\em Journal of Approximation Theory }{\bf 189} (2015) 76-80.
\bibitem{peelman} S. Peelman, J. van der Herten, M. De Vos, W-s Lee, S. Van Huffel and  A. Cuyt, Sparse reconstruction of correlated multichannel activity,
{\em IEEE Engineering in medicine and biology society conference proceedings} (2013) 3897-3900.
\bibitem{potts} D. Potts and M. Tasche. Parameter estimation for nonincreasing exponential sums by Prony-like methods, {\em Linear Algebra and its Applications} {\bf 439} (2013) 1024-1039.
\bibitem{rump} S.M. Rump, Structured Perturbations Part 1: Normwise Distances, {\em SIAM Journal of Matrix Analysis and Applications}, {\bf 25} (2003) 1-30.
\bibitem{Sch} J.L. Schiff, Normal Families, Springer Verlag (1993).
\bibitem{T13} L. N. Trefethen, Approximation Theory and Approximation Practice, SIAM (2013).
\bibitem{WW} H. Werner and L. Wuytack,  On the continuity of the Pad\'e operator, {\em SIAM J. Numerical Analysis} {\bf 20} (1983) 1273-1280.
\end{thebibliography}
\end{document}